\documentclass[12pt]{article}
\usepackage[top=3cm, bottom=3cm, left=2.55cm, right=2.55cm]{geometry}
\usepackage[tbtags]{amsmath}
\usepackage{amssymb}
\usepackage{amsthm}
\usepackage{arydshln}
\usepackage{fancyhdr}
\usepackage{latexsym}
\usepackage{mathrsfs}
\usepackage{xcolor}
\usepackage{wasysym}
\usepackage{fancyhdr}
\usepackage{float}
\usepackage{graphicx}
\usepackage[numbers,sort&compress]{natbib}
\usepackage{pict2e,color}
\usepackage{tikz}
\usepackage{tkz-graph}
\allowdisplaybreaks[4]

\newtheorem{theorem}{\bf Theorem}[section]
\newtheorem{lemma}[theorem]{Lemma}

\newtheorem{corollary}[theorem]{Corollary}

\newtheorem{conjecture}[theorem]{Conjecture}

\newtheorem{claim}[theorem]{Claim}

\begin{document}
\title{\bf\Large MaxCut in graphs with sparse neighborhoods}
\date{}
\author{Jinghua Deng\footnote{Email: jinghua\_deng@163.com}, ~Jianfeng Hou\footnote{Research partially supported by NSFC (Grant No. 12071077). Email: jfhou@fzu.edu.cn}, ~Siwei Lin\footnote{Email: linsw0710@163.com}, ~Qinghou Zeng\footnote{Research partially supported by NSFC (Grant No. 12001106), and National Natural Science Foundation of Fujian Province (Grant No. 2021J05128). Email: zengqh@fzu.edu.cn (Corresponding author)}\\
{\small Center for Discrete Mathematics, Fuzhou University, Fujian, 350003, China}}

\maketitle
\begin{abstract}
Let $G$ be a graph with $m$ edges and let $\mathrm{mc}(G)$ denote the size of a largest cut of $G$. The difference $\mathrm{mc}(G)-m/2$ is called the  surplus $\mathrm{sp}(G)$ of $G$. A fundamental problem in MaxCut is to determine $\mathrm{sp}(G)$ for $G$ without specific structure, and the degree sequence $d_1,\ldots,d_n$ of $G$ plays a key role in getting   lower bounds of  $\mathrm{sp}(G)$. A classical example, given by Shearer, is  that $\mathrm{sp}(G)=\Omega(\sum_{i=1}^n\sqrt d_i)$ for triangle-free graphs $G$, implying that $\mathrm{sp}(G)=\Omega(m^{3/4})$. It was extended to graphs with sparse neighborhoods by  Alon, Krivelevich and Sudakov. In this paper, we establish a novel and stronger result for a more general family of graphs with sparse neighborhoods.

Our result can  derive many  well-known  bounds on surplus of $H$-free graphs for different $H$, such as  triangles,  even cycles,  graphs having a  vertex whose removal makes them acyclic, or  complete bipartite graphs $K_{s,t}$ with $s\in \{2,3\}$. It can also deduce many new  (tight) bounds on $\mathrm{sp}(G)$ in $H$-free graphs $G$ when $H$ is any graph having a  vertex whose removal results in a  bipartite graph with relatively small Tur\'{a}n number, especially the even wheel. This contributes to a conjecture raised by Alon, Krivelevich and Sudakov. Moreover, we obtain new families of graphs $H$ such that $\mathrm{sp}(G)=\Omega(m^{3/4+\epsilon(H)})$ for some constant $\epsilon(H)>0$ in $H$-free graphs $G$, giving evidences to a conjecture suggested by Alon, Bollob\'as, Krivelevich and Sudakov.

\medskip

\textbf{2020 Mathematics Subject Classification}: primary 05C35; secondary 05C38

\textbf{Keywords}: MaxCut, semidefinite programming, $H$-free graph, bipartite graph
\end{abstract}
\section{Introduction}\label{SEC:Introduction}
The celebrated \emph{MaxCut} problem is to determine the number of edges, denoted by $\mathrm{mc}(G)$, in a largest bipartite subgraph of a graph $G$. It is a central problem in discrete mathematics and theoretical computer science, and has received a lot of attention in the last 50 years, both from an algorithmic perspective in theoretical computer science and from an extremal perspective in combinatorics.

Using a probabilistic argument or a greedy algorithm, it is easy to show that $\mathrm{mc}(G)\ge m/2$ for every graph $G$ with $m$ edges.  In 1967, Erd\H{o}s \cite{Erd1967} demonstrated that the factor 1/2 cannot be improved in general, even for graphs containing no short cycles. Therefore, a natural problem is to estimate the error term $\mathrm{mc}(G)-m/2$ of a graph $G$, which we call \emph{surplus} and denote by $\mathrm{sp}(G)$. A classical result of Edwards \cite{E1973, Ed1975} asserts that every graph $G$ with $m$ edges has surplus
\[
\mathrm{sp}(G)\ge\frac{\sqrt{8m+1}-1}{8},
\]
and this is tight for complete graphs with an odd number of vertices.

Although the bound $\Omega(\sqrt m)$ on the surplus is optimal in general, it can be significantly improved for $H$-\emph{free} graphs, that is, graphs containing no copy of a fixed graph $H$. Given a positive integer $m$ and a fixed graph $H$, let $\mathrm{sp}(m,H)$ denote the smallest possible value of $\mathrm{sp}(G)$ over all $H$-free graphs $G$ with $m$ edges. The study of the surplus of $H$-free graphs was initiated by Erd\H{o}s and Lov\'asz in the 70s, and has received significant attention since then.  Alon, Bollob\'as,  Krivelevich and Sudakov \cite{ABKS2003} suggested the following general conjecture.
\begin{conjecture}[Alon, Bollob\'as, Krivelevich and Sudakov \cite{ABKS2003}]\label{General-Graph}
For any fixed graph $H$, there is a constant $\epsilon(H)>0$ such that
\[
\mathrm{sp}(m,H)=\Omega(m^{3/4+\epsilon(H)}).
\]
\end{conjecture}
In fact, we do not even know whether there exists an absolute constant $\alpha>1/2$ such that $\mathrm{sp}(m,H)=\Omega(m^{\alpha})$ for every fixed graph $H$ and all $m$. We refer the reader to \cite{ CKL2021, ZH2017} for partial results on this conjecture. The best known lower bound is due to Glock, Janzer and Sudakov \cite{GJS2021}, who proved that $\mathrm{sp}(m,K_r)=\Omega(m^{1/2+3/(4r-2)})$ for all $m$.

A  more difficult conjecture raised by Alon, Krivelevich and Sudakov \cite{AKS2005} is to determine the order of magnitude of $\mathrm{sp}(m,H)$ for all $H$.
\begin{conjecture}[Alon, Krivelevich and Sudakov \cite{AKS2005}]\label{Tight-Conj}
For any fixed graph $H$, there is a constant $c(H)>0$ such that
\[
\mathrm{sp}(m,H)=\Theta(m^{c(H)}).
\]
\end{conjecture}
The case when $H$ is a triangle has received particularly much attention. Erd\H{o}s and Lov\'asz (see \cite{Erd1979}) firstly proved that $\mathrm{sp}(m,K_3)=\Omega(m^{{2/3}+o(1)})$. This is further improved by Poljak and Tuza \cite{PT1994}. A major breakthrough  is due to Shearer \cite{She1992}, who proved that
\[
\mathrm{sp}(G)=\Omega\left(\sum_{v\in V(G)}\sqrt{d_G(v)}\right).
\]
This bound is tight in general and implies that $\mathrm{sp}(m,K_3)=\Omega(m^{3/4})$. Finally, based on Shearer's result, Alon \cite{A1996} proved that $\mathrm{sp}(m,K_3)=\Theta(m^{4/5})$. His result is exemplary for many of the challenges and developed methods in the area. Alon, Krivelevich and Sudakov \cite{AKS2005} further extended Shearer's argument and proved a similar result for graphs with sparse neighborhoods.

For a graph $G$, let $d_G(v)$ denote the \emph{degree} of the vertex $v$ in $G$. For positive real numbers $c$ and $\epsilon$, say  that $G$ is $(c,\epsilon)$-\emph{sparse} if for every $v\in V(G)$ the subgraph induced by all the $d_G(v)$ neighbors of $v$ contains at most $c\,d_G(v)^{2-\epsilon}$ edges.
\begin{theorem}[Alon, Krivelevich and Sudakov \cite{AKS2005}]\label{AKS}
There exist two positive constants $c$ and $\delta$ with the following property. If $G$ is  $(c,1/2)$-sparse, then
\[
\mathrm{sp}(G)\ge\delta\sum_{v\in V(G)}\sqrt{d_G(v)}.
\]
\end{theorem}
Based on this result, Alon, Krivelevich and Sudakov \cite{AKS2005} established tight bounds on the magnitude of $\mathrm{sp}(m,H)$ for various graphs $H$, such as graphs obtained by connecting a single vertex to all the vertices of  nontrivial forests,  even cycles $C_{2k}$ for $k\in\{2,3,5\}$ or  complete bipartite graphs $K_{s,t}$ for $s\in\{2,3\}$. For more problems and results in this area, we refer the reader to \cite{BS2002,FHM2023,HW2021,LZ2021,LZ2022,RHZ2022,WL2022,ZH2018}. We also direct the reader to \cite{CFK2019,DFH2023} for related problems in hypergraphs.

In this paper, we extend Theorem \ref{AKS}  to all $(c,\epsilon)$-sparse graphs and get the following novel and stronger version.


\begin{theorem}\label{more general}
For any constant $\epsilon\in[0,1]$ and any $c>0$, there exist two positive constants $\delta_1=\delta_1(c)$ and $\delta_2=\delta_2(c)$ with the following property. If $G$ is  $(c,\epsilon)$-sparse, then
\begin{align}\label{main-them-spG}
\mathrm{sp}(G)\ge\delta_1\sum_{v\in V(G)}d_G(v)^{\tau}+\delta_2\sum_{uv\in E(G)}\frac{(d_G(u)d_G(v))^\tau}{|V(G)|},
\end{align}
where $\tau:=\tau(\epsilon)=\min\{\epsilon,1/2\}$. This is tight up to the values of $\delta_1$ and $\delta_2$ for each $\epsilon\in[0,1/3]$.
\end{theorem}
\noindent{\bf Remarks.}

\begin{itemize}
\item  Theorem \ref{AKS} can be deduced from Theorem \ref{more general} by letting  $\epsilon=1/2$.

\item The second term in the right hand of the inequality \eqref{main-them-spG} is useful when $G$ is relatively dense.

\item The positive constant $c$ can be  arbitrary  large in Theorem \ref{more general}.
\end{itemize}


 Note that a graph $G$ is $d$-\textit{degenerate} if every subgraphs of $G$ contains a vertex of degree at most $d$. Theorem \ref{more general} implies the following result with respect to the degeneracy of a graph, from which we can derive many well-known bounds and new (tight) bounds on $\mathrm{sp}(m,H)$ for various graphs $H$ (see Section \ref{Applications}).
\begin{theorem}\label{Sparse-Dense}
Let $\epsilon\in[0,1]$ and $c>0$ be real numbers,   and let $\tau=\min\{\epsilon,1/2\}$. Suppose that $G$ is  $(c,\epsilon)$-sparse with $m$ edges. Then the following statements  hold.
\begin{itemize}
  \item[$(\mathrm{i})$] If $G$ is $d$-degenerate, then $\mathrm{sp}(G)=\Omega(m/d^{1-\tau})$.
  \item[$(\mathrm{ii})$] If $G$ has average degree $d$, then $\mathrm{sp}(G)=\Omega(d^{1+2\tau})$.
\end{itemize}
\end{theorem}
\noindent{\bf Remarks}.
\begin{itemize}
\item Theorem \ref{Sparse-Dense} $(\mathrm{i})$ provides a partial result of a conjecture raised by Carlson, Kolla, Li, Mani, Sudakov and Trevisan \cite{CKL2021}, which states that any $H$-free $d$-degenerate graph with $m$ edges has surplus $\Omega(m/\sqrt d)$.

\item Theorem \ref{Sparse-Dense} $(\mathrm{ii})$  can be viewed as a local version of a recent result given by Glock, Janzer and Sudakov \cite{GJS2021}, which says that every graph with average degree $d$ and less than $(1-\varrho)d^3/6$ triangles for some $\varrho>0$ has surplus $\Omega(d^2)$.
\end{itemize}

The rest of the paper is organized as follows. The  remainder of this section provides some definitions and notations.   In Section \ref{Applications}, we present a series of applications of Theorem \ref{Sparse-Dense} to $H$-free graphs.  We prove Theorems \ref{more general} and \ref{Sparse-Dense} in Section \ref{Main-Result-Proof}. In Section \ref{SEC:Remarks}, we give some concluding remarks.

\noindent{\bf Notations}: All graphs considered are finite, simple and undirected. Let $G$ be a graph. We denote $e(G)$ the number of edges of $G$. A \emph{cut} $(A,B)$ of $G$, is bipartition of $V(G)$ satisfying $A\cup B=V(G)$ and $A\cap B=\emptyset$. An edge $e$ is called a \emph{cut edge} of $(A,B)$ if $|e\cap A|=|e\cap B|=1$. The \emph{size} of $(A,B)$ is the number of cut edges of $(A, B)$.  For $S\subseteq V(G)$, let $G[S]$ denote the  induced subgraph of $G$ by $S$. For $v\in V(G)$, $N_G(v)$ and $d_G(v)$  stand for the neighbourhood and degree of $v$, respectively.  We occasionally omit subscripts if they are clear from the context. For a positive integer $n$, write $[n]:=\{1,\ldots,n\}$. Denote by $\|\cdot\|$ the Euclidean norm.

\section{Applications of Theorem \ref{Sparse-Dense}}\label{Applications}
In this section, we list many applications of  Theorem \ref{Sparse-Dense} to $H$-free graphs for some specific graphs $H$. Surprisingly, Tur\'{a}n type theory of graphs plays a key role in our arguments. 

Let $H$ be a graph. The \emph{Tur\'{a}n number} of  $H$, denoted by $\text{ex}(n,H)$, is the maximum number of edges in an $n$-vertex $H$-free graph. The study of $\mathrm{ex}(n, H)$ is perhaps the central topic in extremal graph  theory. One of the most famous results in this regard is Tur\'{a}n's theorem, which states that for every integer $\ell \ge 2$ the Tur\'{a}n number $\mathrm{ex}(n,K_{\ell+1})$ is uniquely achieved by the complete balanced $\ell$-partite graph on $n$ vertices.  By the classical Erd\H{o}s-Stone-Simonovits
theorem \cite{ES66,ES46}, we have
\[
\text{ex}(n,H)=\left(1-\frac{1}{\chi(H)-1}+o(1)\right)n^2,
\]
where $\chi(H)$ is the chromatic number of $H$. Therefore, the order of $\text{ex}(n,H)$ is known, unless $H$ is a bipartite graph. One of the major open
problems in extremal graph theory is to understand the function $\text{ex}(n,H)$ when $H$ is a bipartite graph. We refer the reader to \cite{AKS2003, BC2018, FS2013, ST2020} for more details.

We now present the following general result obtained from Theorem \ref{Sparse-Dense}, which makes our applications as immediate corollaries.

\begin{theorem}\label{chi32}
Let $\epsilon\in[0,1]$ be a constant and $\tau=\min\{\epsilon,1/2\}$. If $H$ is a graph that has a vertex whose removal results in a bipartite graph $B$ with $\mathrm{ex}(n,B)=O(n^{2-\epsilon})$, then
\begin{align}\label{Chi3}
\mathrm{sp}(m,H)=\Omega\left(m^{(1+2\tau)/(2+\tau)}\right).
\end{align}
Moreover, if $\mathrm{ex}(n,H)=O(n^{2-\alpha})$ for some $\alpha\in[0,1]$, then
\begin{align}\label{Chi2}
\mathrm{sp}(m,H)=\Omega\left(m^{\tau+(1-\tau)/(2-\alpha)}\right).
\end{align}
\end{theorem}
\noindent{\bf Remarks}.
\begin{itemize}
\item The lower bound \eqref{Chi2} is better than  \eqref{Chi3} if $\alpha>\tau/(1+\tau)$. We believe that \eqref{Chi2} is more useful than \eqref{Chi3} when $H$ is a bipartite graph.

\item  We  give an alternative proof of \eqref{Chi2} without using Theorem \ref{Sparse-Dense} in Appendix.
\end{itemize}

\begin{proof}[\bf Proof]
Let $G$ be an $H$-free graph with $m$ edges. By the definition of $B$, we know that $G$ is a $(c,\epsilon)$-sparse graph for some constant $c>0$ in view of $\mathrm{ex}(n,B)=O(n^{2-\epsilon})$.

We first prove the lower bound \eqref{Chi3}.  Let $d=\Theta(m^{1/(2+\tau)})$. If $G$ is $d$-degenerate, then
\[
\mathrm{sp}(G)=\Omega(m/d^{1-\tau})=\Omega\left(m^{(1+2\tau)/(2+\tau)}\right)
\]
by  Theorem \ref{Sparse-Dense} $(\mathrm{i})$. Otherwise, there is an induced subgraph $G'$ of $G$ with minimum degree at least $d$. Clearly, $G'$ is $(c,\epsilon)$-sparse with average degree at least $d$ as $G'$ is also $H$-free and $\mathrm{ex}(n,B)=O(n^{2-\epsilon})$. It follows from Theorem \ref{Sparse-Dense} $(\mathrm{ii})$ that
$$\mathrm{sp}(G)\ge\mathrm{sp}(G')=\Omega(d^{1+2\tau})=\Omega\left(m^{(1+2\tau)/(2+\tau)}\right).$$
So, \eqref{Chi3} holds.

In the following, we establish \eqref{Chi2}. Indeed, if $G$ is $d$-degenerate with $d=\Theta(m^{(1-\alpha)/(2-\alpha)})$, then
\[
\mathrm{sp}(G)=\Omega(m/d^{1-\tau})=\Omega\left(m^{\tau+(1-\tau)/(2-\alpha)}\right)
\]
by  Theorem \ref{Sparse-Dense} $(\mathrm{i})$. We are done.  Otherwise, $G$ contains a subgraph $G'$ with minimum degree at least $d$. Note that the number of vertices of $G'$ is $N\leq2m/d=\Theta(d^{1/(1-\alpha)})$ as $d=\Theta(m^{(1-\alpha)/(2-\alpha)})$. It follows that $d=\Omega(N^{1-\alpha})$. Thus, the number of edges of $G'$ is
\[
e(G')\geq\frac12dN=\Omega(N^{2-\alpha}).
\]
Since $G'$ is $H$-free and $\mathrm{ex}(n,H)=O(n^{2-\alpha})$ for some $\alpha\in[0,1]$, we know that $e(G')=O(N^{2-\alpha})$. This leads to a contradiction by choose a proper constant factor in the notation $d=\Theta(m^{(1-\alpha)/(2-\alpha)})$.
\end{proof}

\subsection{Known results revisited and some new tight results}

Theorem \ref{chi32} is very useful. It not only  deduces some existing results, but also gets many new results. In what follows, we present some reasonable bounds derived immediately from Theorem \ref{chi32}, all of which contribute to Conjectures \ref{General-Graph} and \ref{Tight-Conj}.

The first optimal result of MaxCut in $H$-free graphs  is due to Alon \cite{A1996}, who proved that $\mathrm{sp}(m,K_3)=\Theta(m^{4/5})$. Alon,  Krivelevich and Sudakov \cite{AKS2005} generalized this and established an interesting result. See also \cite{BJS2023} for a modified version of the following theorem.
\begin{theorem}[Alon,  Krivelevich and Sudakov \cite{AKS2005}]\label{graph-removal-one-forest}
If $H$ is a graph obtained by connecting a single vertex to all vertices of a fixed nontrivial forest, then
\[
\mathrm{sp}(m,H)=\Theta\left(m^{4/5}\right).
\]
\end{theorem}

Using Theorem \ref{chi32} with $\tau=\epsilon=1/2$ in \eqref{Chi3}, we can strengthen Theorem \ref{graph-removal-one-forest} as follows.
\begin{theorem}\label{max-cut-chromatic-3}
If $H$ is a graph that has a vertex whose removal results in a nontrivial bipartite graph $B$ with $\mathrm{ex}(n,B)=O(n^{3/2})$, then
\[
\mathrm{sp}(m,H)=\Omega\left(m^{4/5}\right).
\]
This is tight for any $H$ containing a copy of $K_3$ or $K_{3,3}$.
\end{theorem}
\noindent{\bf Remarks}. Two well-known examples in this flavor are $K_3$ and $K_{3,t}$ for any $t\ge3$, whose original proof can be found in \cite{A1996} and \cite{AKS2005}, respectively. Their extremal constructions are also suitable for our theorem. A typical new example of Theorem \ref{max-cut-chromatic-3} is the even wheel $W_{2k}$ obtained by connecting a vertex to all the vertices of a cycle of length $2k$.
\begin{corollary}\label{max-cut-wheel}
For any integer $k\ge 2$, we have
\[
\mathrm{sp}(m,W_{2k})=\Theta\left(m^{4/5}\right).
\]
\end{corollary}

In fact, we have a more general class of graphs. For two integers $k\ge 3$ and $t\ge 1$, let $C_k$ denote the cycle of length $k$ and let $tC_k$ denote the disjoint union of $t$ copies of $C_k$.  Recently, Hou, Hu, Li, Liu, Yang and Zhang \cite{HHL2023} considered the Tur\'an numbers of disjoint union of bipartite graphs, and the following is a typical corollary of their main result.
\begin{theorem}[Hou, Hu, Li, Liu, Yang and Zhang \cite{HHL2023}]\label{THM:C2k}
For any fixed $k \geq 2$, there exist constants $N_0$ and $c_{k} > 0$ such that for all integers $n \geq N_0$ and $t \in [0,c_{k} \cdot \frac{{\rm{ex}}(n,C_{2k})}{n}]$, we have
\begin{align*}
    \mathrm{ex}\left(n, (t+1)C_{2k}\right)
        = \binom{n}{2} - \binom{n-t}{2} + \mathrm{ex}(n-t, C_{2k}).
\end{align*}
\end{theorem}

Combining Theorems \ref{max-cut-chromatic-3} and \ref{THM:C2k}, we have the following corollary.
\begin{corollary}
For any fixed $t\ge 1$ and $k \ge 2$, if $H$ is a graph that has a vertex whose removal results in $tC_{2k}$, then
\[
\mathrm{sp}(m,H)=\Theta\left(m^{4/5}\right).
\]
\end{corollary}

Alon,  Krivelevich and Sudakov \cite{AKS2005} also proved an elegant result for even cycles.
\begin{theorem}[Alon,  Krivelevich and Sudakov \cite{AKS2005}]\label{even-cycle-AKS}
For any integer $k\ge 2$, we have
\[
\mathrm{sp}(m,C_{2k})=\Omega\left(m^{(2k+1)/(2k+2)}\right).
\]
This is tight up to a constant factor for $k\in\{2,3,5\}$.
\end{theorem}

The famous Bondy-Simonovits theorem \cite{BS1974} shows that the number of edges in an $n$-vertex $C_{2k}$-free graph is at most $100n^{1+1/k}$. Note also that any graph containing no certain fixed nontrivial forest can span only a linear number of edges.  These together with Theorem \ref{chi32} by choosing $\alpha=1-1/k$ and $\tau=1/2$ ($\epsilon=1$) in the lower bound \eqref{Chi2} imply Theorem \ref{even-cycle-AKS}. Actually, Theorem \ref{even-cycle-AKS} is a special case of a recent result observed by Balla, Janzer and Sudakov \cite{BJS2023}. The result is not explicitly stated in \cite{AKS2005}, but can be obtained straightforwardly using their method.
\begin{theorem}[Balla, Janzer and Sudakov \cite{BJS2023}]\label{max-cut-chromatic-2}
If $H$ is a graph that has a vertex whose removal makes the graph acyclic and its Tur\'an number satisfies $\mathrm{ex}(n,H)=O(n^{1+\alpha})$ for some constant $\alpha$, then
\[
\mathrm{sp}(m,H)=\Omega\left(m^{(\alpha+2)/(2\alpha+2)}\right).
\]
\end{theorem}

It is easy to check that Theorem \ref{max-cut-chromatic-2} can be obtained by letting $\tau=1/2$ ($\epsilon=1$) in  \eqref{Chi2} of Theorem \ref{chi32}.
Note that this result is tight for $K_{2,t}$ by the Erd\H{o}s-R\'enyi graph (see \cite{AKS2005}).

\subsection{Surplus with exponent at least 3/4}
In this section, based on Theorem \ref{chi32}, we give new families of $H$-free graphs such that $\mathrm{sp}(m,H)=\Omega(m^{3/4+\epsilon(H)})$ for some constant $\epsilon(H)>0$. We hope that our results could shed some light on solving Conjecture \ref{General-Graph}. 

Applying Theorem \ref{chi32} with $\epsilon>2/5$ in  \eqref{Chi3}, we have the following result.
\begin{theorem}\label{8/5}
For any constant $\epsilon\in(2/5,1]$, if $H$ is a graph that has a vertex whose removal results in a bipartite graph $B$ with $\mathrm{ex}(n,B)=O(n^{2-\epsilon})$, then there exists a constant $\epsilon(H)>0$ such that
\[
\mathrm{sp}(m,H)=\Omega\left(m^{3/4+\epsilon(H)}\right).
\]
\end{theorem}

We give some examples to show that there do exist some graphs $H$ satisfying Theorem \ref{8/5}. We first describe a family of graphs $B$ with $\mathrm{ex}(n,B)=\Theta(n^{8/5})$. For any integer $t\ge2$, let $H_{2,t}$ be a graph obtained from two vertex disjoint copies of $K_{2,t}$ by adding a matching that joins the two images of every vertex in $K_{2,t}$. Jiang, Ma and Yepremyan \cite{JMY2022} proved that
\[
\mathrm{ex}(n,H_{2,t})=\Theta(n^{8/5})
\]
for sufficiently large $t$, and $\mathrm{ex}(n,H_{2,t})=O(n^{8/5})$ for any $t\ge2$. In particular, the 3-dimensional cube $Q_8$ is isomorphic to $H_{2,2}$ and $\mathrm{ex}(n,Q_8)=O(n^{8/5})$. More generally, Kang, Kim and Liu \cite{KKL2021} proved that there exists a family of graphs $H_t$ such that for any $t\ge2$
\[
\mathrm{ex}(n,H_t)=\Theta(n^{2-t/(2t+1)})=O(n^{8/5}).
\]

\begin{corollary}
For any integer $t\ge2$ and $B\in\{H_{2,t},H_t\}$, let $H\supset B$ be a graph with $|B|+1$ vertices. Then there exists a constant $\epsilon_t\ge0$ such that
\[
\mathrm{sp}(m,H)=\Omega(m^{3/4+\epsilon_t}).
\]
\end{corollary}

Applying Theorem \ref{chi32} with $\epsilon\le2/5$, we can also find graphs $H$ satisfying Conjecture \ref{General-Graph}. In this situation, we need extra restrictive assumption on the Tur\'an number of $H$ and the value of $\alpha$. Using \eqref{Chi2} with simple calculations, we have the following result.
\begin{theorem}\label{8/5+}
For any constants $\epsilon\in(0,2/5]$ and $\alpha>(2-4\epsilon)/(3-4\epsilon)$, let $H$ be a graph with $\mathrm{ex}(n,H)=O(n^{2-\alpha})$. If $H$ has a vertex whose removal results in a bipartite graph $B$ with $\mathrm{ex}(n,B)=\Theta(n^{2-\epsilon})$, then there exists a constant $\epsilon(H)>0$ such that
\[
\mathrm{sp}(m,H)=\Omega\left(m^{3/4+\epsilon(H)}\right).
\]
\end{theorem}
\noindent{\bf Remarks}. Observe that $B\subseteq H$. This means that any $B$-free graph is also $H$-free, implying that $\alpha\le\epsilon$. This together with our assumption $\alpha>(2-4\epsilon)/(3-4\epsilon)$ yields that $$0.3596\approx(7-\sqrt{17})/8<\epsilon\le2/5.$$

\tikzset{every picture/.style={line width=0.75pt}} 
\begin{figure}[htbp]
\centering
\begin{tikzpicture}[x=0.55pt,y=0.55pt,yscale=-0.75,xscale=0.75]

\draw    (366.62,200.69) -- (366.86,288.28) ;
\draw   (366.07,200.44) .. controls (368.31,200.64) and (370.27,198.89) .. (370.45,196.53) .. controls (370.64,194.17) and (368.98,192.11) .. (366.74,191.91) .. controls (364.51,191.72) and (362.55,193.47) .. (362.36,195.82) .. controls (362.18,198.18) and (363.84,200.25) .. (366.07,200.44) -- cycle ;
\draw    (361.46,200.73) -- (179.97,286.3) ;
\draw    (369.62,200.69) -- (451.09,290.7) ;
\draw    (364.07,200.44) -- (271.98,286.54) ;
\draw    (370.62,197.69) -- (541.94,290.51) ;
\draw    (179.85,295.08) -- (179.25,383.58) ;
\draw   (179.3,294.83) .. controls (181.54,295.03) and (183.5,293.27) .. (183.68,290.92) .. controls (183.87,288.56) and (182.21,286.5) .. (179.97,286.3) .. controls (177.74,286.11) and (175.78,287.86) .. (175.59,290.21) .. controls (175.41,292.57) and (177.07,294.64) .. (179.3,294.83) -- cycle ;
\draw    (179.3,294.83) -- (150.45,383.58) ;
\draw    (179.85,295.08) -- (212.37,384.34) ;
\draw    (269.85,294.32) -- (269.25,382.82) ;
\draw   (269.31,294.07) .. controls (271.54,294.27) and (273.5,292.52) .. (273.69,290.16) .. controls (273.87,287.8) and (272.21,285.74) .. (269.98,285.54) .. controls (267.74,285.35) and (265.78,287.1) .. (265.6,289.45) .. controls (265.41,291.81) and (267.07,293.88) .. (269.31,294.07) -- cycle ;
\draw    (367.42,297.74) -- (367.97,383.58) ;
\draw   (366.87,297.49) .. controls (369.11,297.68) and (371.07,295.93) .. (371.25,293.57) .. controls (371.44,291.22) and (369.78,289.15) .. (367.54,288.96) .. controls (365.31,288.76) and (363.35,290.51) .. (363.16,292.87) .. controls (362.98,295.22) and (364.64,297.29) .. (366.87,297.49) -- cycle ;
\draw    (366.87,297.49) -- (338.45,383.58) ;
\draw    (367.42,297.74) -- (398.21,384.34) ;
\draw  [color={rgb, 255:red, 0; green, 0; blue, 0 }  ,draw opacity=1 ][fill={rgb, 255:red, 0; green, 0; blue, 0 }  ,fill opacity=1 ] (149.78,392.11) .. controls (152.01,392.31) and (153.97,390.55) .. (154.16,388.2) .. controls (154.34,385.84) and (152.68,383.78) .. (150.45,383.58) .. controls (148.21,383.39) and (146.25,385.14) .. (146.07,387.49) .. controls (145.88,389.85) and (147.54,391.92) .. (149.78,392.11) -- cycle ;
\draw  [fill={rgb, 255:red, 74; green, 74; blue, 74 }  ,fill opacity=1 ] (178.91,391.64) .. controls (181.15,391.84) and (183.11,390.08) .. (183.3,387.73) .. controls (183.48,385.37) and (181.82,383.31) .. (179.59,383.11) .. controls (177.35,382.92) and (175.39,384.67) .. (175.2,387.02) .. controls (175.02,389.38) and (176.68,391.45) .. (178.91,391.64) -- cycle ;
\draw  [fill={rgb, 255:red, 74; green, 74; blue, 74 }  ,fill opacity=1 ] (213.5,390.97) .. controls (215.74,391.17) and (217.7,389.42) .. (217.88,387.06) .. controls (218.07,384.71) and (216.41,382.64) .. (214.17,382.44) .. controls (211.94,382.25) and (209.97,384) .. (209.79,386.35) .. controls (209.6,388.71) and (211.27,390.78) .. (213.5,390.97) -- cycle ;
\draw  [fill={rgb, 255:red, 74; green, 74; blue, 74 }  ,fill opacity=1 ] (240.78,390.35) .. controls (243.02,390.55) and (244.98,388.8) .. (245.16,386.44) .. controls (245.35,384.08) and (243.69,382.02) .. (241.45,381.82) .. controls (239.22,381.63) and (237.26,383.38) .. (237.07,385.73) .. controls (236.89,388.09) and (238.55,390.16) .. (240.78,390.35) -- cycle ;
\draw  [fill={rgb, 255:red, 74; green, 74; blue, 74 }  ,fill opacity=1 ] (268.2,390.88) .. controls (270.43,391.08) and (272.4,389.33) .. (272.58,386.97) .. controls (272.77,384.61) and (271.11,382.55) .. (268.87,382.35) .. controls (266.64,382.16) and (264.67,383.91) .. (264.49,386.26) .. controls (264.3,388.62) and (265.96,390.69) .. (268.2,390.88) -- cycle ;
\draw  [fill={rgb, 255:red, 74; green, 74; blue, 74 }  ,fill opacity=1 ] (291.52,387.11) .. controls (291.7,389.34) and (293.75,390.99) .. (296.1,390.78) .. controls (298.45,390.57) and (300.22,388.59) .. (300.04,386.36) .. controls (299.86,384.12) and (297.81,382.48) .. (295.46,382.68) .. controls (293.1,382.89) and (291.34,384.87) .. (291.52,387.11) -- cycle ;
\draw  [fill={rgb, 255:red, 74; green, 74; blue, 74 }  ,fill opacity=1 ] (337.71,392.11) .. controls (339.94,392.31) and (341.9,390.55) .. (342.09,388.2) .. controls (342.28,385.84) and (340.61,383.78) .. (338.38,383.58) .. controls (336.14,383.39) and (334.18,385.14) .. (334,387.49) .. controls (333.81,389.85) and (335.47,391.92) .. (337.71,392.11) -- cycle ;
\draw  [fill={rgb, 255:red, 74; green, 74; blue, 74 }  ,fill opacity=1 ] (366.84,391.64) .. controls (369.08,391.84) and (371.04,390.08) .. (371.23,387.73) .. controls (371.41,385.37) and (369.75,383.31) .. (367.52,383.11) .. controls (365.28,382.92) and (363.32,384.67) .. (363.13,387.02) .. controls (362.95,389.38) and (364.61,391.45) .. (366.84,391.64) -- cycle ;
\draw  [fill={rgb, 255:red, 74; green, 74; blue, 74 }  ,fill opacity=1 ] (397.83,390.97) .. controls (400.07,391.17) and (402.03,389.42) .. (402.21,387.06) .. controls (402.4,384.71) and (400.74,382.64) .. (398.5,382.44) .. controls (396.27,382.25) and (394.31,384) .. (394.12,386.35) .. controls (393.94,388.71) and (395.6,390.78) .. (397.83,390.97) -- cycle ;
\draw  [fill={rgb, 255:red, 74; green, 74; blue, 74 }  ,fill opacity=1 ] (450.75,294.96) .. controls (452.99,295.16) and (454.95,293.41) .. (455.14,291.05) .. controls (455.32,288.7) and (453.66,286.63) .. (451.43,286.43) .. controls (449.19,286.24) and (447.23,287.99) .. (447.04,290.35) .. controls (446.86,292.7) and (448.52,294.77) .. (450.75,294.96) -- cycle ;
\draw  [fill={rgb, 255:red, 74; green, 74; blue, 74 }  ,fill opacity=1 ] (541.6,294.77) .. controls (543.84,294.97) and (545.8,293.22) .. (545.99,290.86) .. controls (546.17,288.51) and (544.51,286.44) .. (542.27,286.24) .. controls (540.04,286.05) and (538.08,287.8) .. (537.89,290.16) .. controls (537.71,292.51) and (539.37,294.58) .. (541.6,294.77) -- cycle ;
\draw   (150.28,399.25) .. controls (150.36,403.91) and (152.73,406.2) .. (157.4,406.12) -- (172.21,405.85) .. controls (178.87,405.73) and (182.24,408) .. (182.33,412.67) .. controls (182.24,408) and (185.53,405.61) .. (192.2,405.49)(189.2,405.55) -- (207.01,405.23) .. controls (211.67,405.14) and (213.96,402.77) .. (213.88,398.1) ;
\draw   (239.92,400.38) .. controls (240.01,405.05) and (242.38,407.34) .. (247.05,407.26) -- (261.85,406.99) .. controls (268.52,406.87) and (271.89,409.14) .. (271.98,413.81) .. controls (271.89,409.14) and (275.18,406.75) .. (281.85,406.63)(278.85,406.69) -- (296.65,406.37) .. controls (301.32,406.28) and (303.61,403.91) .. (303.52,399.24) ;
\draw   (335.33,399.25) .. controls (335.42,403.91) and (337.79,406.2) .. (342.45,406.12) -- (357.26,405.85) .. controls (363.93,405.73) and (367.3,408) .. (367.38,412.67) .. controls (367.3,408) and (370.59,405.61) .. (377.25,405.49)(374.25,405.55) -- (392.06,405.23) .. controls (396.73,405.14) and (399.02,402.77) .. (398.93,398.1) ;
\draw   (160,453) .. controls (159.98,457.67) and (162.3,460.01) .. (166.97,460.03) -- (266.97,460.45) .. controls (273.64,460.48) and (276.96,462.82) .. (276.94,467.49) .. controls (276.96,462.82) and (280.3,460.5) .. (286.97,460.53)(283.97,460.52) -- (386.97,460.95) .. controls (391.64,460.97) and (393.98,458.65) .. (394,453.98) ;
\draw   (441,304) .. controls (440.96,308.67) and (443.27,311.02) .. (447.94,311.06) -- (487.44,311.41) .. controls (494.11,311.47) and (497.42,313.83) .. (497.38,318.5) .. controls (497.42,313.83) and (500.77,311.53) .. (507.44,311.59)(504.44,311.56) -- (546.94,311.94) .. controls (551.61,311.99) and (553.96,309.68) .. (554,305.01) ;
\draw    (267.64,292.81) -- (241.12,386.09) ;
\draw    (295.78,386.73) -- (271.64,292.81) ;

\draw (178.84,424.86) node [anchor=north west][inner sep=0.75pt]    {$s$};
\draw (269.57,424.86) node [anchor=north west][inner sep=0.75pt]    {$s$};
\draw (365.33,425.62) node [anchor=north west][inner sep=0.75pt]    {$s$};
\draw (270.01,476.43) node [anchor=north west][inner sep=0.75pt]    {$t$};
\draw (490.34,325.42) node [anchor=north west][inner sep=0.75pt]    {$s'$};
\end{tikzpicture}
\caption{$T_{s,t,s'}$}   \label{fig1}
\end{figure}

In what follows, we describe a family of graphs $H$, initially studied by Jiang, Jiang and Ma \cite{JJM2022}, satisfying the assumption of Theorem \ref{8/5+}. For three nonnegative integers $s,t,s'$, let $T_{s,t,s'}$ denote the tree defined in Figure \ref{fig1}. Define the $p$th power $T_{s,t,s'}^p$ of $T_{s,t,s'}$  by taking $p$ disjoint copies of $T_{s,t,s'}$ and identifying the different copies of the vertex $v$ for each $v$ in black as depicted in Figure \ref{fig1}. Jiang, Jiang and Ma \cite{JJM2022} first studied the Tur\'an exponent for this family of graphs. For convenience, we employ a refined version of theirs result, which established by Conlon and Janzer \cite{CJ2022}.
\begin{theorem}[Conlon and Janzer \cite{CJ2022}]\label{Turan-Exponent}
For any integers $s'\ge1$, $s\ge s'-1$, $t\ge s-s'+1$ and $p\ge1$, we have
\[
\mathrm{ex}(n,T^p_{s,t,s'})=O\left(n^{2-\frac{t+1}{st+t+s'}}\right).
\]
In particular, if $s'=0$, then we need $t\ge s+2$. Moreover, there exists $p\ge1$ such that
\[
\mathrm{ex}(n,T^p_{s,t,s'})=\Omega\left(n^{2-\frac{t+1}{st+t+s'}}\right).
\]
\end{theorem}

Now, we define the graphs $H$ and $B$ as follows. Let $B=T^p_{2,t,0}$ for each $t\in\{5,6,\ldots,11\}$. By Theorem \ref{Turan-Exponent}, there exists $p\ge1$ such that $\mathrm{ex}(n,B)=\Theta(n^{2-\epsilon})$, where
\[
\epsilon=\frac{t+1}{3t}\in[0.36,0.4].
\]
Fix such an integer $p$, and let $H=T^p_{2,t,1}$. Again, by Theorem \ref{Turan-Exponent}, we have
\[
\mathrm{ex}(n,H)=O\left(n^{2-\frac{t+1}{3t+1}}\right).
\]
It is easy to check that $\alpha=(t+1)/(3t+1)\ge(2-4\epsilon)/(3-4\epsilon)$ for each $5\le t\le11$, where the equality holds if and only if $t=11$.
\begin{corollary}
There exist an integer $p\ge1$ and a constant $\epsilon_t\ge0$ such that
\[
\mathrm{sp}(m,T^p_{2,t,1})=\Omega(m^{3/4+\epsilon_t})
\]
for each $t\in\{5,6,\ldots,11\}$, where $\epsilon_t=0$ if and only if $t=11$.
\end{corollary}
\noindent{\bf Remarks}. We mention that there also exist graphs $H$ and $B$ satisfying Theorem \ref{8/5+} by choosing other values of $s,t,s'$. Here, we do not list all of them.

\subsection{$K_{s,t}$-free graphs}
In this section, we focus on a special case of Conjecture \ref{General-Graph} and give a bound on $\mathrm{sp}(m,K_{s,t})$. Alon, Krivelevich and Sudakov \cite{AKS2005} suggested the following conjecture on $K_{s,t}$-free graphs.
\begin{conjecture}[Alon, Krivelevich and Sudakov \cite{AKS2005}]\label{Kst}
For all $t\ge s\ge2$ and all $m$, we have
\[
\mathrm{sp}(m,K_{s,t})=\Omega\left(m^{\frac34+\frac{1}{8s-4}}\right).
\]
\end{conjecture}
Alon, Krivelevich and Sudakov \cite{AKS2005} confirmed this conjecture for $s\in\{2,3\}$. It remains wide open for any $s\ge4$. We remark that if Conjecture \ref{Kst} is true, then this is tight at least for all $t\ge(s-1)!+1$, as shown by the projective norm graphs. The conjecture is widely open. In fact  we  don't know whether there is an absolute constant $\alpha>1/2$ such that $\mathrm{sp}(m,K_{s,t})=\Omega(m^{\alpha})$.

A well-known result of K\H ov\'ari, S\'os and Tur\'an \cite{KST1954} asserts that any $n$-vertex $K_{s,t}$-free graph contains at most $O(n^{2-1/s})$ edges. This together with Theorem \ref{chi32} gives the current best known lower bound of Conjecture \ref{Kst} by setting $\tau=\min\{1/(s-1),1/2\}$ and $\alpha=1/s$ in \eqref{Chi2}.
\begin{theorem}\label{main thm}
For all $t\ge s\ge3$ and all $m$, we have
\begin{align*}
\mathrm{sp}(m,K_{s,t})=\Omega\left(m^{\frac 12+\frac 3{4s-2}}\right).
\end{align*}
In particular, $\mathrm{sp}(m,K_{2,t})=\Omega\left(m^{5/6}\right)$.
\end{theorem}
\noindent{\bf Remarks}. Our result matches Conjecture \ref{Kst} for $s\in\{2,3\}$. Coincidentally, our theorem also matches a recent result of  Glock, Janzer and Sudakov \cite{GJS2021} on $K_s$-free graphs, which states that $\mathrm{sp}(m,K_s)=\Omega(m^{1/2+3/(4s-2)})$ for all $m$. It would be interesting to get an improvement on both results, even just for the case $s=4$.


\section{Proofs of Theorems \ref{more general} and \ref{Sparse-Dense}}\label{Main-Result-Proof}
In this section, we prove Theorems \ref{more general} and \ref{Sparse-Dense}. To estimate the surplus of graphs via semidefinite programming, we need the following lemma essentially due to Goemans and Williamson \cite{GW1995}, whose short proof (see, e.g. \cite{GJS2021}) is included here for the sake of completeness.

\begin{lemma}[Glock, Janzer and Sudakov \cite{GJS2021}]\label{sdp lemma}
Let $N$ be a positive integer, and let $G$ be a graph.
Then, for any set of non-zero vectors $\{\mathbf x^v:v\in V(G)\}\subset\mathbb R^N$, we have
\[
\mathrm{sp}(G)\ge-\frac1{\pi}\sum_{uv\in E(G)}\arcsin\left(\frac{\langle\mathbf x^u,\mathbf x^v\rangle}{\|\mathbf x^u\|\|\mathbf x^v\|}\right).
\]
\end{lemma}
\begin{proof}[\bf Proof.]
    Let $\mathbf{z}$ be a uniformly random unit vector in $\mathbb{R}^N$, let $A=\{v\in V(G):\langle\mathbf x^v,\mathbf z\rangle\ge0\}$ and $B=V(G)\backslash A$. Given an edge $uv\in E(G)$, the angle between $\mathbf x^v$ and $\mathbf x^u$ is $\arccos\left({\frac{\langle\mathbf x^v,\mathbf x^u\rangle}{\|\mathbf x^u\|\|\mathbf x^v\|}}\right)$. Thus, the probability that $uv$ lies in the cut $(A,B)$ is
    \begin{align*}
        \mathbb P(\text{$uv$ is a  cut edge of $(A,B)$})=\frac{1}{\pi}\arccos\left({\frac{\langle\mathbf x^v,\mathbf x^u\rangle}{\|\mathbf x^u\|\|\mathbf x^v\|}}\right)=\frac12-\frac1\pi \arcsin\left({\frac{\langle\mathbf x^v,\mathbf x^u\rangle}{\|\mathbf x^u\|\|\mathbf x^v\|}}\right).
    \end{align*}
    By the linearity of expectation, we conclude that the expected size of the cut $(A,B)$ is at least $$\frac{e(G)}2-\frac1\pi \sum_{uv\in E(G)}\arcsin\left({\frac{\langle\mathbf x^v,\mathbf x^u\rangle}{\|\mathbf x^u\|\|\mathbf x^v\|}}\right),$$
    from which the result follows.
\end{proof}

The following result on regular graphs is available to show the tightness of Theorem \ref{more general}. See, e.g., \cite{A1996, ABKS2003} for a proof.
\begin{lemma}\label{Eigen}
Let $G$ be an regular graph with $n$ vertices, and let $\lambda$ be the smallest eigenvalue of (the adjacency matrix of) $G$. Then
\[
\mathrm{sp}(G)\le-\lambda n/4.
\]
\end{lemma}

We also use the following two folklore lemmas.
\begin{lemma}\label{Degenerate}
Let $G$ be a $d$-degenerate graph on $n$ vertices. Then there is an ordering $v_1,\ldots,v_n$ of the vertices of $G$ such that for every $1\leq i\leq n$ the vertex $v_i$ has at most $d$ neighbours $v_j$ with $j<i$.
\end{lemma}
\begin{lemma}\label{folklore}
Every graph with average degree $d$ has a subgraph with minimum degree at least $d/2$.
\end{lemma}

The next well-known lemma asserts that it suffices to find a subgraph with relatively large surplus in order to show that a graph has large surplus. See, e.g., \cite{A1996, AKS2005} for a proof.
\begin{lemma}\label{Induced-Sub}
Let $G$ be a graph and $U\subseteq V(G)$. Then
\[
\mathrm{sp}(G)\ge\mathrm{sp}(G[U]).
\]
\end{lemma}

Now we are in a position to bound the surplus of graphs with sparse neighborhoods.
\begin{proof}[\bf{Proof of Theorem \ref{more general}}]
Let $G$ be a $(c,\epsilon)$-sparse graph defined on the vertex set  $[n]:=\{1,\ldots,n\}$.  For every vertex $i\in [n]$, let $d_i$ denote the degree of $i$. For each edge $ij\in E(G)$, let $d_{ij}=|N(i)\cap N(j)|$. Now we define a vector $\mathbf x^i=(x^i_1,x^i_2,\ldots,x^i_n)\in\mathbb R^n$ for every vertex $i\in [n]$ by taking
\begin{align*}
x_j^i=
\begin{cases}
1+\rho d_i^{\tau}/n, &\text{if $j=i$},\\
-\rho d_i^{\tau-1}, &\text{if $j\in N(i)$},\\
\rho d_i^{\tau}/n, &\text{otherwise},
\end{cases}
\end{align*}
where $\tau:=\tau(\epsilon)=\min\{\epsilon,1/2\}$ and $\rho:=\rho(c)=\min\{c/32,1/(32c)\}\le1/32$. To employ Lemma \ref{sdp lemma}, we aim to bound $\|\mathbf x^i\|$ for each $i\in [n]$ and $\langle\mathbf x^i,\mathbf x^j\rangle$ for each edge $ij\in E(G)$.

\begin{claim}\label{norm}
For each $i\in [n]$, we have
\begin{align*}
1\le\|\mathbf x^i\|\le\sqrt2.
\end{align*}
\end{claim}
\begin{proof}
For each $i\in [n]$, it follows from the definition of $\mathbf x^i$ that
\begin{align*}
\|\mathbf x^i\|&=\sqrt{\left(1+\frac{\rho d_i^{\tau}}{n}\right)^2+d_i\left(-\rho d_i^{\tau-1}\right)^2+(n-d_i-1)\left(\frac{\rho d_i^{\tau}}{n}\right)^2}\\
&=\sqrt{1+\frac{2\rho d_i^{\tau}}{n}+\rho^2d_i^{2\tau-1}\left(1+\frac{d_i}{n}-\frac{d_i^2}{n^2}\right)}.
\end{align*}
Clearly, $\|\mathbf x^i\|\ge1$. Since $d_i\le n$ and $\tau\le1/2$, we have
\[
d_i^{2\tau-1}\le1 \;\:\text{and}\;\: 1+\frac{d_i}{n}-\frac{d_i^2}{n^2}\le2.
\]
This together with the choice of $\rho$ yields that
\begin{align*}
\|\mathbf x^i\|\le\sqrt{1+2\rho^2}\le\sqrt2,
\end{align*}
completing the proof of Claim \ref{norm}.
\end{proof}
\begin{claim}\label{inner product}
For each edge $ij\in E(G)$, we have
\begin{align*}
\langle\mathbf x^i,\mathbf x^j\rangle\le-\rho\left(d_i^{\tau-1}+d_j^{\tau-1}\right)-\frac{\rho^2(d_id_j)^\tau}{n}+4\rho^2(d_i d_j)^{\tau-1}d_{ij}.
\end{align*}
\end{claim}
\begin{proof}
For each edge $ij\in E(G)$, let $S_1=(N(i)\cup N(j))\setminus(N(i)\cap N (j))$ and $S_2=[n]\setminus S_1$. Clearly, we have
\begin{align}\label{S12}
\langle\mathbf x^i,\mathbf x^j\rangle&=\sum_{k\in[n]}x_k^ix_k^j=\sum_{k\in S_1}x_k^ix_k^j+\sum_{k\in S_2}x_k^ix_k^j.
\end{align}
Now, we bound them one by one.

For each edge $ij\in E(G)$, by the definition of $\mathbf x^i$, it is easy to see that
\[
x_i^ix_i^j=-\left(1+\frac{\rho d_i^{\tau}}n\right)\rho d_j^{\tau-1}  \;\;\text{and}\;\;  x_j^ix_j^j=-\rho d_i^{\tau-1}\left(1+\frac{\rho d_j^{\tau}}n\right).
\]
Moreover, we have $x_k^ix_k^j=-\rho^2{d_i^{\tau-1}d_j^{\tau}}/n$ for each $k\in N(i)\setminus N[j]$, and $x_k^ix_k^j=-\rho^2{d_i^{\tau}d_j^{\tau-1}}/n$ for each $k\in N(j)\setminus N[i]$. Note that $|N(i)\setminus N[j]|=d_i-d_{ij}-1$ and $|N(j)\setminus N[i]|=d_j-d_{ij}-1$. It follows that
\begin{align}\label{S1}
\sum_{k\in S_1}x_k^ix_k^j=&\left(\sum_{k\in \{i,j\}}+\sum_{k\in N(i)\setminus N[j]}+\sum_{k\in N(j)\setminus N[i]}\right)x_k^ix_k^j\notag\\
=&-\left(1+\frac{\rho d_i^{\tau}}n\right)\rho d_j^{\tau-1}-\rho d_i^{\tau-1}\left(1+\frac{\rho d_j^{\tau}}n\right)\notag\\
&+(d_i-d_{ij}-1)\left(-\frac {\rho^2 d_i^{\tau-1}d_j^{\tau}}n\right)+(d_j-d_{ij}-1)\left(-\frac {\rho^2 d_i^{\tau}d_j^{\tau-1}}n\right)\notag\\
=&-\rho d_i^{\tau-1}-\rho d_j^{\tau-1}-\frac{\rho^2(d_id_j)^\tau}{n}\left(2-\frac{d_{ij}}{d_i}-\frac{d_{ij}}{d_j}\right).
\end{align}
We also have $x_k^ix_k^j=\rho^2(d_id_j)^{\tau-1}$ for each $k\in N(i)\cap N (j)$, and $x_k^ix_k^j=\rho^2{(d_id_j)^\tau}/{n^2}$ for each $k\in [n]\setminus (N(i)\cup N(j))$. Observe that $|N(i)\cap N(j)|=d_{ij}$ and $|[n]\setminus(N(i)\cup N(j))|=n-d_i-d_j+d_{ij}$. Hence
\begin{align}\label{S2}
\sum_{k\in S_2}x_k^ix_k^j&=\left(\sum_{k\in N(i)\cap N(j)}+\sum_{k\in [n]\setminus (N(i)\cup N(j))}\right)x_k^ix_k^j\notag\\
&=\rho^2(d_id_j)^{\tau-1}d_{ij}+(n-d_i-d_j+d_{ij})\frac{\rho^2(d_id_j)^\tau}{n^2}\notag\\
&=\rho^2(d_id_j)^\tau\left(\frac{d_{ij}}{d_id_j}+\frac{d_{ij}}{n^2}+\frac{n-d_i-d_j}{n^2}\right).
\end{align}
Combining equalities \eqref{S12}, \eqref{S1} and \eqref{S2}, we conclude that
\begin{align*}
\langle\mathbf x^i,\mathbf x^j\rangle
&=-\rho d_i^{\tau-1}-\rho d_j^{\tau-1}+\rho^2(d_id_j)^\tau\left(d_{ij}\left(\frac{1}{n}+\frac1{d_i}\right)\left(\frac{1}{n}+\frac1{d_j}\right)-\frac{n+d_i+d_j}{n^2}\right)\\
&\le-\rho\left(d_i^{\tau-1}+d_j^{\tau-1}\right)-\frac{\rho^2(d_id_j)^\tau}{n}+4\rho^2(d_i d_j)^{\tau-1}d_{ij},
\end{align*}
where the last inequality follows from $d_i\le n$ for each $i\in[n]$. This completes the proof of Claim \ref{inner product}.
\end{proof}

For any $x\in[-1,1]$ with $x\le b-a$ for some $a,b\ge0$, we have
\[
\arcsin(x)\le\frac{\pi}2b-a.
\]
In fact, if $x < 0$, then $\arcsin(x)\le x\le b-a\le\frac{\pi}2b-a$; and if $x\ge0$, then $\arcsin(x)\le\frac{\pi}2x\le\frac{\pi}2(b-a)\le\frac{\pi}2b-a$. Recall that $1\le\|\mathbf x^i\|\le\sqrt2$ for each $i\in[n]$ by Claim \ref{norm}. These together with Claim \ref{inner product} imply that
\begin{align*}
\arcsin\left(\frac{\langle\mathbf x^i,\mathbf x^j\rangle}{\|\mathbf x^i\|\|\mathbf x^j\|}\right)&\le-\frac{\rho\left(d_i^{\tau-1}+d_j^{\tau-1}\right)}{\|\mathbf x^i\|\|\mathbf x^j\|}-\frac{\rho^2(d_id_j)^\tau}{n\|\mathbf x^i\|\|\mathbf x^j\|}+\frac{2\pi \rho^2(d_i d_j)^{\tau-1}d_{ij}}{\|\mathbf x^i\|\|\mathbf x^j\|}\\
&\le-\frac{\rho}{2}\left(d_i^{\tau-1}+d_j^{\tau-1}\right)-\frac{\rho^2(d_id_j)^\tau}{2n}+2\pi \rho^2(d_i d_j)^{\tau-1}d_{ij}
\end{align*}
for every edge $ij\in E(G)$. It follows from Lemma \ref{sdp lemma} that
\begin{align}\label{goal inequality}
\mathrm{sp}(G)&\ge-\frac1{\pi}\sum_{ij\in E(G)}\arcsin\left(\frac{\langle\mathbf x^i,\mathbf x^j\rangle}{\|\mathbf x^i\|\|\mathbf x^j\|}\right)\notag\\
&\ge\frac{\rho}{8}\sum_{ij\in E(G)}\left(d_i^{\tau-1}+d_j^{\tau-1}\right)+\sum_{ij\in E(G)}\frac{\rho^2(d_id_j)^\tau}{8n}-2\rho^2\sum_{ij\in E(G)}(d_i d_j)^{\tau-1}d_{ij}\notag\\
&=\frac{\rho}{8}\sum_{i\in[n]}d_i^{\tau}+\sum_{ij\in E(G)}\frac{\rho^2(d_id_j)^\tau}{8n}-2\rho^2\sum_{ij\in E(G)}(d_i d_j)^{\tau-1}d_{ij}.
\end{align}

In what follows, it suffices to bound $\sum_{ij\in E(G)}(d_i d_j)^{\tau-1}d_{ij}$.
\begin{claim}\label{sparse neighborhood}
\begin{align*}
\sum_{ij\in E(G)}(d_i d_j)^{\tau-1}d_{ij}\le c\sum_{i\in[n]}d_i^{\tau}.
\end{align*}
\end{claim}
\begin{proof}
Note that $d_i^{2\tau-2}+d_j^{2\tau-2}\ge2(d_i d_j)^{\tau-1}$ for each edge $ij\in E(G)$. This implies that
\begin{align*}
\sum_{ij\in E(G)}(d_i d_j)^{\tau-1}d_{ij}&\le\sum_{ij\in E(G)}\frac12\left(d_i^{2\tau-2}+d_j^{2\tau-2}\right)d_{ij}\\
&=\frac12\sum_{i\in[n]} \left(d_i^{2\tau-2}\sum_{j\in N(i)}d_{ij}\right).
\end{align*}
Recall that the induced subgraph on all the $d_i$ neighbors of vertex number $i$ contains at most $cd_i^{2-\epsilon}$ edges for each $i\in[n]$ by our assumption. Since the degree of the vertex $j\in N(i)$ in the subgraph $G[N(i)]$ is exact $d_{ij}$, we conclude that
\[
\sum_{j\in N(i)}d_{ij}=2e(G[N(i)])\le2cd_i^{2-\epsilon}\le 2cd_i^{2-\tau}.
\]
Hence
\begin{align*}
\sum_{ij\in E(G)}(d_i d_j)^{\tau-1}d_{ij}\le\frac12\sum_{i\in[n]}\left(d_i^{2\tau-2}\sum_{j\in N(i)}d_{ij}\right)\le c\sum_{i\in[n]} d_i^{\tau},
\end{align*}
completing the proof of Claim \ref{sparse neighborhood}.
\end{proof}

Combining the inequality \eqref{goal inequality} and Claim \ref{sparse neighborhood}, we deduce that
\begin{align*}
\mathrm{sp}(G)&\ge\left(\frac{\rho}{8}-2\rho^2c\right)\sum_{i\in[n]}d_i^{\tau}+\sum_{ij\in E(G)}\frac{\rho^2(d_id_j)^\tau}{8n}\\
&\ge\frac{\rho}{16}\sum_{i\in[n]}d_i^{\tau}+\sum_{ij\in E(G)}\frac{\rho^2(d_id_j)^\tau}{8n}.
\end{align*}
The last inequality holds in view of $\rho c\le1/32$ by the choice of $\rho$. Thus, we complete the lower bound of Theorem \ref{more general} by setting $\delta_1=\rho/16$ and $\delta_2=\rho^2/8$.

In what follows, we show that the inequality \eqref{main-them-spG} is tight up to the values of $\delta_1$ and $\delta_2$ for each $\epsilon\in[0,1/3]$. Clearly, this is true for $\epsilon=0$ by Edwards' bound. We may assume that $\epsilon\in(0,1/3]$. Consider the following graph $G=(V,E)$ constructed by Delsarte and Goethals, and by Turyn (see \cite{KS2006}). Let $q$ be a prime power and let $V$ be the elements of the two-dimensional vector space over $GF(q)$. This means that $G$ contains $n=q^2$ vertices. Partition the $q+1$ lines through the origin of the space into two sets $P$ and $N$, where $|P| = k$. Two vertices $x$ and $y$ of $G$ are adjacent if and only if $x-y$ is parallel to a line in $P$. Note that $G$ is $d=k(q-1)$-regular, and its smallest eigenvalue is $-k$. It follows from Lemma \ref{Eigen} that
\begin{align}\label{Eigen-Upper}
\mathrm{sp}(G)=O(kq^2).
\end{align}
For each $\epsilon\in(0,1/3]$, choose $k=\Theta(q^{\epsilon/(1-\epsilon)})$. This implies that
\begin{align}\label{d-regular}
d=\Theta(kq)=\Theta(k^{1/\epsilon})=\Theta(q^{1/(1-\epsilon)}).
\end{align}
Moreover, $G$ is actually a strongly regular graph and the number of common neighbours of each pair of adjacent vertices of $G$ is exactly $q-2+(k-1)(k-2)$. It is easy to deduce that the induced subgraph $G_v$ on all the $d$ neighbors of the vertex $v\in V$ has
\begin{align}\label{Gv}
e(G_v)=\Theta(d(q+k^2))= \Theta(d^{2-\epsilon}+d^{1+2\epsilon})=\Theta(d^{2-\epsilon})
\end{align}
edges. Hence, we conclude that $G$ is $(c,\epsilon)$-sparse for some positive constant $c$. In view of \eqref{main-them-spG} and \eqref{d-regular}, we have
\begin{align}\label{large-surplus}
\mathrm{sp}(G)=\Omega\left(nd^\epsilon+d^{1+2\epsilon}\right)=\Omega\left(kq^{2}+kq^{(1+\epsilon)/(1-\epsilon)}\right)=\Omega(kq^2).
\end{align}
This together with \eqref{Eigen-Upper} yields that $\mathrm{sp}(G)=\Theta(kq^2)$, completing the proof of Theorem \ref{more general}.
\end{proof}

\begin{proof}[{\bf Proof of Theorem \ref{Sparse-Dense}}]
$(\mathrm{i})$ Since $G$ is $d$-degenerate, it follows from Lemma \ref{Degenerate} that there exists a labelling $v_1,\ldots,v_n$ of the vertices of $G$ such that $d^+_i\leq d$ for every $i$, where $d^+_i$ denotes the number of neighbors $v_j$ of $v_i$ with $j<i$. Clearly, $\sum_{i=1}^{n}d^+_i=m$. By Theorem \ref{more general}, there exists a constant $\delta_1>0$ such that
\begin{eqnarray*}
\mathrm{sp}(G)\ge\delta_1\sum_{i=1}^{n}d_i^\tau\geq\delta_1\sum_{i=1}^{n}(d_{i}^{+})^\tau \geq\frac{\delta_1\sum_{i=1}^{n} d_{i}^{+}}{d^{1-\tau}}=\delta_1\;\frac{m}{d^{1-\tau}}.
\end{eqnarray*}
This completes the first claim of Theorem \ref{Sparse-Dense}.

$(\mathrm{ii})$ Since $G$ has average degree $d$, it follows from Lemma \ref{folklore} that there exists a subgraph $G'$ of $G$ with minimum degree at least $d/2$. By Theorem \ref{more general} and Lemma \ref{Induced-Sub}, there exists a constant $\delta_2>0$ such that
\begin{align*}
\mathrm{sp}(G)\ge\mathrm{sp}(G')&\ge\delta_2\sum_{uv\in E(G')}\frac{(d_{G'}(u)d_{G'}(v))^\tau}{|V(G')|}\\
&\ge\frac12\delta_2\left(\frac d2\right)^{1+2\tau}=\frac{\delta_2}{4^{1+\tau}}\;d^{1+2\tau}.
\end{align*}
Thus, we complete the second claim of Theorem \ref{Sparse-Dense}.
\end{proof}

\section{Concluding remarks}\label{SEC:Remarks}
Extended a result of Alon, Krivelevich and Sudakov \cite{AKS2005}, we give a lower bound for surplus of graphs with sparse neighborhoods using  the semidefinite programming method. It can get many old and new results about the MaxCut of $H$-free graphs if $H$  is a graph that has a vertex whose removal results in a bipartite graph.  New ideas should be involved to bound the surplus of $K_r$-free graphs when $r\ge 4$.  The first nontrivial result on this topic was given by Zeng and Hou \cite{ZH2017}, who proved that $\mathrm{sp}(m, K_r)=\Omega(m^{(r-1)/(2r-3)+o(1)})$. This was improved to $\Omega(m^{(r+1)/(2r-1)})$ by Glock, Janzer and Sudakov \cite{GJS2021}. As noted by Glock, Janzer and Sudakov \cite{GJS2021}, any improvement on it, even just beating the exponent $5/7$ in the $r=4$ case, would be interesting. Another interesting  open problem is to decide whether there exists a positive absolute constant $\epsilon$ such that any $K_r$-free graph with $m$ edges has surplus $\Omega(m^{1/2+\epsilon})$.

An easy observation \cite{A1996} shows that   $\mathrm{sp}(G)\ge m/(2\chi)$ for any graph $G$ with $m$ edges and chromatic number $\chi$. Thus, a fundamental method in studying MaxCut is to find an induced subgraph with many edges and small chromatic number in view of Lemma \ref{Induced-Sub}. It has been used widely to study the MaxCut of graphs without specific cycles \cite{A1996,ABKS2003,AKS2005,ZH2017,ZH2018}. The second term on the right hand of \eqref{main-them-spG} in Theorem \ref{more general} plays the same role. Thus, the results listed in Section \ref{Applications} do not have tedious structural analysis. We believe that this will be helpful to study the MaxCut of other $H$-free graphs.

We can get almost all known results on MaxCut of $H$-free graphs through Theorem \ref{chi32}, except $H$ is an odd cycle. For any odd $r\ge 3$, Glock, Janzer and Sudakov \cite{GJS2021} (see also \cite{BJS2023}) proved that $\mathrm{sp}(m,C_r)=\Theta(m^{(r+1)/(r+2)})$, confirming a conjecture posed by Alon, Krivelevich and Sudakov \cite{AKS2005}. Their proof presents many new ideas, which together with Theorem \ref{more general} will certainly help us for further research in this area. 

Recall that Theorem \ref{more general} is tight for $\epsilon\in[0,1/3]$. 
If $\epsilon>1/3$, then it follows from \eqref{Gv} and \eqref{large-surplus} that $e(G_v)=\Theta(d^{1+2\epsilon})=\omega(d^{2-\epsilon})$ and
\[
nd^\epsilon+d^{1+2\epsilon}=\Omega\left(kq^{(1+\epsilon)/(1-\epsilon)}\right)=\omega(kq^2)=\omega(\mathrm{sp}(G)).
\]
This means that if the induced subgraph on all the $d$ neighbors of some vertex contains more than $\omega(d^{2-\epsilon})$ edges, then the assertion of Theorem \ref{more general} may cease to hold. It is interesting to decide whether Theorem \ref{more general} is tight for $\epsilon>1/3$.

\section*{Appendix}
In this appendix, we give an alternative proof of the lower bound \eqref{Chi2} concerning the surplus in a graph. In fact, it suffices to prove the following theorem according to Theorem \ref{more general}.
\begin{theorem}\label{Sum}
Let $G$ be a graph with $m$ edges. Suppose that there exist two constants $c>0$ and $1\le\alpha\le2$ such that $e(S)\le c|S|^{\alpha}$ for any $S\subseteq V(G)$. Then for each $\tau\in(0,1)$
\[
\sum_{i\in V(G)}d_G(v)^\tau=\Omega\left(m^{\tau+\frac{1-\tau}{\alpha}}\right).
\]
\end{theorem}
To attack this theorem, we employ the famous Young's inequality as follows.
\begin{lemma}[Young's Inequality]\label{Young}
Let $a,b,p,q$ be four positive real numbers. If $p>1$ and
\[
\frac1p+\frac1q=1,
\]
then
\[
\frac{a}{p}+\frac{b}{q}\ge a^{1/p}b^{1/q},
\]
with equality holds if and only if $a=b$.
\end{lemma}
\begin{proof}[\bf Proof of Theorem \ref{Sum}]
Let $V(G)=\{v_1,\ldots,v_n\}$ and $d_i=d_G(v_i)$ for each $i\in [n]$. We may assume that $d_1\geq d_2\geq \ldots \geq d_n$. Let $H$ be the subgraph induced by $\{v_1 \dots v_t\}$, where $t<n$ will be chosen later. Clearly, $e(H)\le ct^\alpha$ by our assumption. This implies that $\sum_{i > t} d_i \geq m-ct^{\alpha}$. Set $d_t = x$. It follows that
\begin{align*}
\sum_{i=1}^n d_i^{\tau}=\sum_{i\le t} d_i^{\tau}+ \sum_{i>t} d_i^{\tau}&\ge tx^{\tau}+\frac{\sum_{i>t} d_i}{x^{1-\tau}}\\
&\ge tx^{\tau}+\frac{m-c t^\alpha}{x^{1-\tau}}\\
&\ge\left(\frac{tx^\tau}{1-\tau}\right)^{1-\tau}\left(\frac{m-c t^{\alpha}}{\tau x^{1-\tau}}\right)^{\tau},
\end{align*}
where the last inequality holds due to Lemma \ref{Young} by setting $p=1/(1-\tau)$, $q=1/\tau$, $a=tx^\tau/(1-\tau)$ and $b=(m-c t^{\alpha})/(\tau x^{1-\tau})$. Thus, we conclude that
\begin{align*}
\sum_{i=1}^n d_i^{\tau}\ge Ct^{1-\tau}(m-c t^\alpha)^{\tau},
\end{align*}
where $C^{-1}=(1-\tau)^{1-\tau}\tau^\tau$. Choose $t=\left\lfloor\left(\frac{m}{2c}\right)^{1/\alpha}\right\rfloor$. Note that $m\le cn^{\alpha}$ by our assumption. This means that $t\le n/2^{1/\alpha}<n$ as $\alpha\in [1,2]$. Moreover,
\[
t^{1-\tau}(m-c t^\alpha)^{\tau} \ge\left(\frac{1}{2c}-o(1)\right)^\frac{1-\tau}{\alpha}\frac{1}{2^\tau} m^{\tau+\frac{1-\tau}{\alpha}},
\]
implying that
\[
\sum_{i\in [n]}d_i^\tau=\Omega\left(m^{\tau+\frac{1-\tau}{\alpha}}\right).
\]
This completes the proof of Theorem \ref{Sum}.
\end{proof}
\end{document}